\documentclass[12pt]{article}
\usepackage{graphicx}
\usepackage{amsmath}
\usepackage{amssymb}
\usepackage{theorem}

\numberwithin{equation}{section}

\newtheorem{thm}{Theorem}[section]
\newtheorem{lemma}[thm]{Lemma}
\newtheorem{prop}[thm]{Proposition}

{\theorembodyfont{\rmfamily}

\newtheorem{rmk}[thm]{Remark}
}

\newcommand{\qed}{\hfill \mbox{\raggedright \rule{.07in}{.1in}}}
 
\newenvironment{proof}{\vspace{1ex}\noindent{\bf
Proof}\hspace{0.5em}}{\hfill\qed\vspace{1ex}}
\newenvironment{pfof}[1]{\vspace{1ex}\noindent{\bf Proof of
#1}\hspace{0.5em}}{\hfill\qed\vspace{1ex}}

\newcommand{\R}{{\mathbb R}}

\newcommand{\E}{{\mathbb E}}
 \newcommand{\Lip}{\operatorname{Lip}}

\newcommand{\SMALL}{\textstyle}

\newcommand{\dotx}{{\dot x}^{(\epsilon)}}
\newcommand{\doty}{{\dot y}^{(\epsilon)}}
\newcommand{\dotz}{{\dot z}^{(\epsilon)}}
\newcommand{\x}{x^{(\epsilon)}}
\newcommand{\xR}{x^{(\epsilon),R}}
\newcommand{\y}{y^{(\epsilon)}}
\newcommand{\z}{z^{(\epsilon)}}
\newcommand{\taue}{\tau_{\epsilon}}
\newcommand{\taueR}{\tau_{\epsilon,R}}

\title{Homogenization for Deterministic Maps and Multiplicative Noise}

\author{
Georg A. Gottwald\thanks{School of Mathematics and Statistics, University of Sydney, Sydney 2006 NSW, Australia georg.gottwald@sydney.edu.au}
\and
Ian Melbourne\thanks{Mathematics Institute, University of Warwick, Coventry, CV4 7AL, UK  i.melbourne@warwick.ac.uk}
}

\date{28 November 2012, updated 28 April 2015.
\\ \normalsize This paper
corrects an error in the published version of the paper
\\ which appeared in 
{\em Proc.\ Roy.\ Soc. London A} \textbf{469} (2013) 20130201.}

\begin{document}

\maketitle

 \begin{abstract}
A recent paper of Melbourne \& Stuart,
A note on diffusion limits of chaotic skew product flows, \emph{Nonlinearity} {\bf 24} (2011) 1361--1367, gives a rigorous proof of convergence of a fast-slow deterministic system to a stochastic differential equation with additive noise.
In contrast to other approaches, the assumptions on the fast flow are very mild.

In this paper, we extend this result from continuous time to discrete time. Moreover 
we show how to deal with one-dimensional multiplicative noise.   This raises the issue of how to interpret certain stochastic integrals; it is proved that the integrals are of Stratonovich type for continuous time and neither Stratonovich nor It\^o for discrete time.   

We also provide a rigorous derivation of superdiffusive limit where the stochastic differential equation is driven by a stable L\'evy process.   In the case of one-dimensional multiplicative noise, the stochastic integrals are of Marcus type both in the discrete and continuous time contexts.

The original version included an incorrect argument based on the paper of Melbourne \& Stuart.  In accordance with an updated version of the latter, we give the correct argument and 
remove an unnecessary large deviation assumption.
 \end{abstract}


\section{Introduction}

There is considerable interest in understanding how stochastic behaviour emerges from deterministic systems, both in the mathematics and applications literature.   A simple mechanism for emergent stochastic behaviour is via homogenization of multiscale systems, see for example~\cite{PavliotisStuart}.

Recently, Melbourne \& Stuart~\cite{MS11} embarked on a programme to develop a rigorous theory of homogenization based on new ideas in the theory of dynamical systems.  The aim is to avoid excessive mixing assumptions on the fast dynamics, since these are very difficult to establish even for uniformly hyperbolic (Axiom A or Anosov) flows (general references for the ergodic theory of uniformly hyperbolic maps and flows include~\cite{Bowen75,BowenRuelle75,Ruelle78,Sinai72}).   Instead, the theory relies only on relatively mild statistical properties that are known to hold very widely and are independent of mixing assumptions (see~\cite{MS11} or Remark~\ref{rmk-Young} below for further details).

The mechanism for emergent stochastic behaviour in deterministic systems, whereby fast chaotic dynamics induces white noise in the slow variables, is much-studied in the applied literature, see for example~\cite{Beck90,JKRH01,MT00,MTV06,PavliotisStuart}.
See also the program outlined by~\cite{Mackay10}.  The aim here, continuing and extending the work in~\cite{MS11}, is to obtain rigorous results for large classes of fast-slow systems under unusually mild assumptions.

In particular,~\cite{MS11} studied fast-slow ODEs of the form 
\begin{align} \label{eq-MS} \nonumber
\dotx & = \epsilon^{-1}f_0(\y)+f(\x,\y), \quad \x(0)=\xi \\
\doty & = \epsilon^{-2}g(\y), \quad \y(0)=\eta,
\end{align}
where $\x\in\R^d$, $\y\in\R^\ell$, $\epsilon>0$.
It is assumed that
the vector fields $f_0:\R^\ell\to\R^d$, $f:\R^d\times\R^\ell\to\R^d$ and $g:\R^\ell\to\R^\ell$ satisfy certain regularity conditions and that the fast $y$
dynamics possesses a compact attractor $\Lambda\subset\R^\ell$ with ergodic 
invariant probability measure $\mu$ satisfying certain mild chaoticity assumptions.   Finally, it is required that $\int_\Lambda f_0\,d\mu=0$ so that we are in the situation of homogenization rather than averaging.
The conclusion~\cite{MS11} is that $\x\to_w X$ in
$C([0,\infty),\R^d)$ as $\epsilon\to0$, where $X$ is the 
solution to a stochastic differential equation (SDE) of the form 
\begin{align} \label{eq-SDE}
dX=\sqrt\Sigma\, dW+F(X)\,dt, \quad X(0)=\xi.
\end{align}
Here $W$ is unit $d$-dimensional Brownian motion
and $F(x,y)=\int_\Lambda f(x,y)\,d\mu(y)$.
(Throughout, we use $\to_w$ to denote weak convergence in the sense of probability measures~\cite{Billingsley1999}.)

In this paper, we consider the twin goals of (i) allowing multiplicative noise when
$d=1$,
and (ii) proving analogous results for discrete time to those for flows.  

In the introduction, we will focus on the case of discrete 
time.  First we consider the case where there is no multiplicative noise.
Consider the equation
\begin{align} \label{eq-map_add}
\x(n+1)=\x(n)+\epsilon f_0(y(n))+\epsilon^2 f(\x(n),y(n),\epsilon), \quad \x(0)=\xi,
\end{align}
where $\x(n)\in\R^d$ and the fast variables $y(n)$ are generated by a map $g:\R^\ell\to\R^\ell$ with
compact attractor $\Lambda\subset\R^\ell$ and ergodic invariant measure $\mu$.
We require that $\Lambda$ is {\em mildly chaotic} as in~\cite{MS11}.  That is, 
we assume 
a weak invariance principle (WIP).
The precise definitions are recalled in Section~\ref{sec-fast}. A consequence is that
$n^{-1/2}\sum_{j=0}^{n-1}f_0(y(j))$ converges in distribution as $n\to\infty$ to
a $d$-dimensional normal distribution with mean zero and $d\times d$ covariance 
matrix $\Sigma$.

Define $\hat x^{(\epsilon)}(t)=\x(t\epsilon^{-2})$ for
$t=0,\epsilon^2,2\epsilon^2,\ldots$ and linearly interpolate
to obtain $\hat x^{(\epsilon)}\in C([0,\infty),\R^d)$.

Our first main result is a direct analogue of the continuous time 
result of Melbourne \& Stuart~\cite[Theorem~1.1]{MS11}.

\begin{thm}    \label{thm-add_map}
Consider equation~\eqref{eq-map_add}.
Assume that $f_0$ and $f$ are locally Lipschitz in $x$ and $y$, 
and that
$\lim_{\epsilon\to0} f(x,y,\epsilon)=f(x,y,0)$
uniformly on compact subsets of $\R^d\times\Lambda$.
Assume that $\Lambda$ satisfies the WIP as described below.
Suppose that $\int_\Lambda f_0\,d\mu=0$ and 
set $F(x)=\int_\Lambda f(x,y,0)\,d\mu(y)$.

Suppose that solutions $X$ to the SDE~\eqref{eq-SDE} exist on $[0,\infty)$
with probability one.
Then $\hat x^{(\epsilon)}\to_w X$ in $C([0,\infty),\R^d)$ 
as $\epsilon\to0$.
\end{thm}

\begin{rmk}
(a) In~\cite{MS11} it was assumed that $f$ (and hence $F$)
is globally Lipschitz and so the
condition on global existence for $X$ is automatic.   
In the course of the current paper, such global conditions become rather
excessive, so we relax them from the outset, see Subsection~\ref{sec-reg}.

\noindent(b) In the generality of this paper, including Theorem~\ref{thm-add_map}, it is not possible to write down a formula for the covariance matrix $\Sigma$.   However in many situations, including the Young tower situation~\cite{Young98,Young99} discussed
in Section~\ref{sec-fast}, it is possible to prove convergence of second moments~\cite{MTorok12} leading to the expression
\begin{align} \label{eq-moment}
\Sigma=\lim_{n\to\infty}n^{-1}\int_\Lambda
\Bigl(\sum_{j=0}^{n-1} f_0(y(j))\Bigr)
\Bigl(\sum_{j=0}^{n-1} f_0(y(j))\Bigr)^T
\,d\mu.
\end{align}

Under the additional assumption of summable decay of correlations (which is valid for Young towers that satisfy the WIP and are mixing), we obtain the well-known Green-Kubo formula
\begin{align} \label{eq-GK}
\Sigma=\int_\Lambda f_0 f_0^T\,d\mu+
\sum_{n=1}^\infty \int_\Lambda f_0(y(n))f_0(y(0))^T\,d\mu+
\sum_{n=1}^\infty \int_\Lambda f_0(y(0))f_0(y(n))^T\,d\mu.
\end{align}
In particular, we note that~\eqref{eq-moment} is valid for general uniformly hyperbolic maps and~\eqref{eq-GK} is valid under the additional assumption that the map is mixing.

(We have written formula~\eqref{eq-GK} so that it applies equally to invertible and noninvertible maps. Of course in the case of invertible maps 
$\Sigma= \sum_{n=-\infty}^\infty \int_\Lambda f_0(y(n))f_0(y(0))^T\,d\mu$.)
\end{rmk}

Next, we turn to the case of multiplicative noise in the discrete time
setting with $d=1$.  Consider the equation
\begin{align} \label{eq-map_mult}
\x(n+1)=\x(n)+\epsilon h(\x(n))f_0(y(n))+\epsilon^2 f(\x(n),y(n),\epsilon),
\end{align}
where $\x(n)\in\R$ and the fast variables $y(n)$ are generated as before.

As $\epsilon\to0$, we expect that $\hat x^{(\epsilon)}(t)=\x(t\epsilon^{-2})$ converges weakly to solutions $X$ of an SDE of the form 
\[
dX=h(X)\,dW+F(X)\,dt,
\]
but there is the issue of how to interpret the stochastic integral
$\int h(X)\,dW$.
In the continuous time setting, one expects the limiting SDE to be Stratonovich~\cite{Sussmann78, WongZakai65},
and this is indeed the case, see Theorem~\ref{thm-mult}.
However, the discrete time case is very different, as shown by
Givon \& Kupferman~\cite{GivonKupferman04}.  They considered the special case where $h(x)=x$ and $f(x,y,\epsilon)=\lambda x$ ($\lambda$ constant), and showed that in general the stochastic integral is neither It\^{o} nor Stratonovich except in the special case where $y(n)$
is an iid sequence -- in that case the integral is It\^{o}.   Their proof exploited linearity in a crucial way.   Here we extend their results, relaxing linearity and allowing $f$ to depend on $y$.

\begin{thm}    \label{thm-mult_map}
Let $d=1$ and consider equation~\eqref{eq-map_mult}.
Assume that $f_0$ and $f$ are locally Lipschitz in $x$ and $y$, 
and that
$\lim_{\epsilon\to0} f(x,y,\epsilon)=f(x,y,0)$
uniformly on compact subsets of $\R\times\Lambda$.
Assume that $h:\R\to\R$ is $C^1$ and nonvanishing.
Assume that $\Lambda$ satisfies the WIP as described below.
Suppose that $\int_\Lambda f_0\,d\mu=0$ and 
set $F(x)=\int_\Lambda f(x,y,0)\,d\mu(y)$.

Consider the Stratonovich SDE
\begin{align} \label{eq-SDE_mult_map}
dX=\sigma h(X)\circ dW+\Bigl(F(X)-\frac12 h(X)h'(X)\int_\Lambda f_0^2\,d\mu\Bigr)\, dt,
\quad X(0)=\xi,
\end{align}
where $W$ is unit $1$-dimensional Brownian motion.
Suppose that solutions $X$ to the SDE exist on $[0,\infty)$
with probability one.
Then $\hat x^{(\epsilon)}\to_w X$ in $C([0,\infty),\R)$ 
as $\epsilon\to0$.
\end{thm}

\begin{rmk} \label{rmk-GK}
Note as in~\cite{GivonKupferman04} that the correction term in this SDE is It\^o if and only if
$\sigma^2=\int_\Lambda f_0^2\,d\mu$.   This holds in the independent 
case but is generally false.  (For example, the Green-Kubo formula~\eqref{eq-GK} specialised to the case $d=1$ yields
$\sigma^2=\int_\Lambda f_0^2\,d\mu+2\sum_{n=1}^\infty \int_\Lambda f_0(y(n))f_0(y(0))\,d\mu$).

See~\cite[Section~6]{GivonKupferman04} and also~\cite{KupfermanPavliotisStuart04,   PavliotisStuart05} for discussions concerning corrections that are neither It\^o not Stratonovich.
\end{rmk}

The assumption that $h$ is nonvanishing means that we can write $h=1/r'$ where 
$r$ is a monotone differentiable function.   By a change of variables, $Z=r(X)$, it is then possible to reduce to the situation where there is no multiplicative noise, see Sections~\ref{sec-mult_flows} and~\ref{sec-mult_map}.
In the process the drift term $F(X)$ in the SDE is transformed into $\tilde F(Z)$ where 
\[
\tilde F=(F/h)\circ r^{-1},
\]
(see Proposition~\ref{prop-Ito}).

These considerations lead to an extension of Theorem~\ref{thm-mult_map} where
$h$ is not required to be nonvanishing.
Instead we require only that there is a monotone differentiable function $r$
such that $r'=1/h$ and such that $\xi=X(0)$ lies in the domain of
$r$. (It suffices that $h(\xi)\neq0$,
in which case we can choose $r(x)=\int_\xi^x 1/h(y)\,dy$ for $x$ near $\xi$.)
Define $\tilde F=(F/h)\circ r^{-1}$
whenever this expression makes sense.

\begin{prop}  \label{prop-mult_map}
Assume the same set up as in Theorem~\ref{thm-mult_map} except that $h$ is not assumed to be nonvanishing.  Consider the SDE
\begin{align} \label{eq-SDE_mult2_map}
dZ=\sigma \,dW+\tilde F(Z)\,dt, \quad Z(0)=r(\xi).
\end{align}
Suppose that solutions $Z$ to the SDE
 exist on $[0,\infty)$ with probability one.
Then solutions $X$ to the SDE~\eqref{eq-SDE_mult_map}
 exist on $[0,\infty)$ with probability one,
and $\hat x^{(\epsilon)}\to_w X$ in $C([0,\infty),\R)$ 
as $\epsilon\to0$.
\end{prop}

For example, suppose that $h(x)=x$ and $f(x,y)=xq(x,y)$, $F(x)=xQ(x)$.
If $\xi>0$, then we choose $r(x)=\log x$ yielding
$\tilde F(z)=Q(e^z)$ which is well-behaved in many situations
(eg.\ $Q(x)$ constant, $Q(x)=-x^p$, $p\ge0$).  The case $\xi<0$ is similar with
$r(x)=\log(-x)$.

\begin{rmk}[Higher-dimensional multiplicative noise]

A more general setting that includes both situations described above is where
the slow equations have the form
\[
\x(n+1)=\x(n)+\epsilon h(\x(n))f_0(y(n))+\epsilon^2 f(\x(n),y(n),\epsilon), \quad \x(0)=\xi,
\]
where $\x(n)\in\R^d$, $f_0:\Lambda\to\R^d$, $f:\R^d\times\Lambda\times\R\to\R^d$ and $h:\R^d\to L(\R^d,\R^d)$.
Theorem~\ref{thm-add_map} deals with the case $h(x)=I_d$, $d$ general, and Theorem~\ref{thm-mult_map}
covers the case $d=1$, $h$ general.
It is well known that multiplicative noise presents more serious problems in higher dimensions, and it will be necessary to make assumptions beyond the WIP in general.  This is the subject of future work.  However, the methods in this paper generalise immediately to the higher-dimensional situation whenever $h$ has the particular form $h=(Dr)^{-1}$ for some $C^2$
diffeomorphism $r:\R^d\to\R^d$.   The formulas in the discrete case are 
straightforward to derive but unpleasant to write down and we omit them here.
The analogous result in the flow case is described in
Remark~\ref{rmk-higher}.
\end{rmk}

Finally we consider the case where the fast dynamics is not sufficiently 
chaotic to support the central limit theorem.   In this case, we can still hope to prove homogenization theorems but the limiting SDE is driven 
by a stable L\'evy process.   Again, we obtain results for ODES and maps, with additive noise in general dimensions and multiplicative noise for $d=1$.
The interpretation of the stochastic integral in the 
limiting SDE is the same for both continuous and discrete time, and is of Marcus type (see Section~\ref{sec-Levy} for further details.)

The remainder of this paper is organised as follows.
In Section~\ref{sec-fast}, we recall the setting of~\cite{MS11} as regards the chaoticity assumptions on the fast dynamics, though now in the context of discrete time.   In Section~\ref{sec-flows}, we consider fast-slow ODEs, relaxing the uniformity of the Lipschitz conditions in~\cite{MS11} and permitting multiplicative noise for $d=1$ (with a restricted extension to higher dimensions).
In Section~\ref{sec-maps}, we consider fast-slow maps and prove the results stated in this introduction.
In Section~\ref{sec-Levy}, we state and prove results where the limiting SDE is driven by a stable L\'evy process.
In Section~\ref{sec-numerics}, we present some numerical results.
Section~\ref{sec-conc} is a conclusion section.

\section{Assumptions on the fast dynamics}
\label{sec-fast}

In~\cite{MS11}, we made mild assumptions on the fast dynamics that are satisfied by large classes of dynamical systems.   The formulation there is for continuous time.  Here we discuss discrete time (both contexts are required in this paper).

Let $g:\Lambda\to\Lambda$ where $\Lambda$ is a compact subset of $\R^\ell$
and $\mu$ is an ergodic invariant measure supported on $\Lambda$.
Given $y_0=\eta\in\Lambda$, we define the fast variables $y(n)$, $n\ge0$,
by setting $y(n+1)=g(y(n))$.

Let $f_0:\Lambda\to\R^d$ be a Lipschitz observable of
mean zero.  Define $W_n(t)=n^{-\frac12}\sum_{j=0}^{nt-1}f_0(y(j))$ for $t=0,\frac1n,\frac2n,\ldots$ and linearly interpolate to obtain a continuous function
$W_n:[0,\infty)\to\R^d$.  We assume the {\em weak invariance principle (WIP)},
namely that $W_n\to_w \sqrt\Sigma W$ in $C([0,\infty),\R^d)$ where $W$ is unit $d$-dimensional
Brownian motion and $\Sigma$ is a $d\times d$ covariance matrix.


\begin{rmk}  \label{rmk-Young}
As discussed in~\cite[Remark 1.3(a)]{MS11}, the WIP holds for a large class of maps and flows.  These include, but go far beyond, Axiom~A diffeomorphisms and flows, H\'enon-like attractors and Lorenz attractors.
Young~\cite{Young98,Young99} introduced a class of
nonuniformly hyperbolic maps with exponential and polynomial decay of correlations.  For maps, the WIP holds when the correlations are summable.   For flows, it suffices that there is a Poincar\'e map with these properties and then the WIP lifts to these flows (irrespective of the mixing properties of the flow).  Precise statements about the validity of the WIP can be found in~\cite{MN05,MN09}.
\end{rmk}

%
\section{Extensions of the results for flows}
\label{sec-flows}

In this section, we extend the results in~\cite{MS11} by (i) relaxing the global Lipschitz conditions and (ii) allowing multiplicative noise.

\subsection{Relaxing the global Lipschitz condition on $f$}
\label{sec-reg}

Consider the fast-slow system~\eqref{eq-MS}.
We suppose throughout that the fast equations possess a ``mildly chaotic'' 
compact attractor $\Lambda\subset\R^\ell$ satisfying the WIP.
It is natural to assume that $f_0:\Lambda\to\R^d$, $f:\R^d\times\Lambda\to\R^d$ and $g:\Lambda\to\Lambda$ are locally Lipschitz to ensure
existence and uniqueness of solutions to the various initial value problems arising above.
Boundedness and uniformity of Lipschitz constants on $\Lambda$ then follows from
compactness.
However, in~\cite{MS11} it is further assumed (mainly for simplicity) that $f$
is bounded with a uniform Lipschitz constant on the whole of $\R^d\times\Lambda$.

In this subsection, we show that the result of~\cite{MS11} holds without the
global Lipschitz condition provided solutions to the limiting SDE exist
for all time with probability one.   
The formulation of~\cite[Theorem 1.1]{MS11} is unchanged if in addition
solutions to the fast-slow system exist for all time for $\mu$-almost every $\eta$.
 Otherwise, we require the following modification.  Throughout we regard $\xi$ as fixed.

Let $[0,\taue]$ be the maximal interval of existence for a solution $\x$ and define 
\[
\x_*(t)=\begin{cases} \x(t), & 0\le t\le \taue/2 \\
\x(\taue/2), & t\ge \taue/2\end{cases}.
\]
(If $\x$ exists on $[0,\infty)$, then set $\x_*\equiv\x$.)
We say that $\x$ converges weakly to $X$ in $C([0,\infty),\R^d)$
if $\taue\to\infty$ in probability and 
$\x_*$ converges weakly to $X$ in $C([0,\infty),\R^d)$.

\begin{thm} \label{thm-reg}
Assume that the fast equation (with $\epsilon=1$) has a mildly chaotic compact
invariant set $\Lambda$ with invariant ergodic probability measure $\mu$.
Suppose that
$f_0:\Lambda\to\R^d$ and $f:\R^d\times\Lambda\to\R^d$ are locally Lipschitz, 
and that $\int_\Lambda f_0\,d\mu=0$.
Define $F(x)=\int_\Lambda f(x,y)\,d\mu(y)$ and let $\xi\in\R^d$.   

Let $\x,\y$ denote the solutions to the fast-slow system~\eqref{eq-MS}.
Assume that solutions $X$ to the SDE~\eqref{eq-SDE} exist on $[0,\infty)$
with probability one.
Then $\x\to_w X$ in $C([0,\infty),\R^d)$ as $\epsilon\to0$.
\end{thm}

In the remainder of this subsection, we prove Theorem~\ref{thm-reg} by reducing it to the situation in~\cite{MS11}.
The ideas are standard, but care has to be taken since we are talking about
weak convergence of solutions, and solutions to the fast-slow equation can
blow up arbitrarily quickly.

Let $R>0$ and define $f_R:\R^d\times\Lambda\to\R^d$ to be a  globally Lipschitz function that agrees with $f$ on $B_R(\xi)=\{x\in\R^d:|x-\xi|\le R\}$
(where $|x-\xi|$ is Euclidean distance).
Let $\xR$, $\y$ denote solutions to~\eqref{eq-MS} with $f$ replaced by $f_R$.
These solutions exist for all time; in particular
$\xR\in C([0,\infty),\R^d)$.   
Similarly, let $X^R$ denote the solution to the SDE~\eqref{eq-SDE}
with $F$ replaced by $F_R$ where
$F_R(x)=\int_\Lambda f_R(x,y)\,d\mu(y)$.
(Note that $\Sigma$ depends only on $f_0$ and $g$ and hence is independent 
of $R$.)
By~\cite{MS11},
$\xR\to_w X^R$ in $C([0,\infty),\R^d)$ as $\epsilon\to0$.

Next, we define stopping times $\tau_R,\taueR\in(0,\infty)\cup\{\infty\}$
where $\tau_R\ge0$ is least such that $X^R\in \partial B_R(\xi)=\{x\in\R^d:|x-\xi|=R\}$ and
$\taueR\ge0$ is least such that $\xR\in \partial B_R(\xi)$.
By construction, $X^R\equiv X$ on $[0,\tau_R]$ and
$\xR\equiv \x$ on $[0,\taueR]$.

\begin{prop}  \label{prop-R}
$\taue\to\infty$, $\tau_R\to\infty$ and $\taueR\to\infty$ in probability as $R\to\infty$, $\epsilon\to0$. 
\end{prop}

\begin{proof}
By assumption $\tau_R\to\infty$ almost surely, and hence in probability
as $R\to\infty$.   In other words,
for any $T>0$, $\delta>0$, there exists $R$ such that
$P(\tau_R>T)>1-\delta$.

Next, for fixed $R>0$,  given $v\in C([0,\infty],\R)$ we set
$\psi(v)=\sup\{t\ge0:v(t)\in B_R(v(0))\}\in(0,\infty)\cup\{\infty\}$.
This defines a continuous map $\psi$.   
Moreover,  $\taueR=\psi(\xR)$ and $\tau_R=\psi(X^R)$, so it follows from 
the continuous mapping theorem that $\taueR\to_d \tau_R$ as $\epsilon\to0$.

In particular, with $T$ and $R$ as in the first paragraph, we have
that there exists $\epsilon_0>0$ such that
$\mu(\taueR>T)>P(\tau_R>T)-\delta>1-2\delta$ for all $\epsilon\in(0,\epsilon_0)$.
Altogether, we have shown that for any $T>0$, $\delta>0$, there exists
$R$ and $\epsilon_0$ such that 
$\mu(\taueR>T)>1-2\delta$ for all $\epsilon\in(0,\epsilon_0)$, so $\taueR\to\infty$ in probability.

Since $\taue>\taueR$ for all $R$, it follows immediately that
$\taue\to\infty$ in probability.
\end{proof}

\noindent

Now fix $T>0$.   By Proposition~\ref{prop-R}, for any $\delta>0$
there exists $\epsilon_0$ and $R$ such that
$\mu(\taueR>2T)>1-\delta$ for all $\epsilon\in(0,\epsilon_0)$.
We choose $R$ so that in addition $\mu(\tau_R>2T)>1-\delta$.

For this choice of $R$, let $X_T=X^R$ and $\x_T=x^{(\epsilon),R}$.
Then we have defined families of random elements
$X_T$ and $\x_T$ in $C([0,\infty),\R^d)$ such that $\x_T\to_w X_T$.
Moreover, neglecting a set of measure $\delta$, we have
$\taue\ge\taueR>2T$ so $\taue/2>T$ and hence  
 $\x_T\equiv \x\equiv \x_*$ on $[0,T]$.  
Finally, neglecting a set of measure $\delta$, we have $X_T\equiv X$.
Hence $\x_*$ converges weakly to $X$ in $C([0,T),\R^d)$.
Since $T$ is arbitrary,
$\x_*\to_w X$ in $C([0,\infty),\R^d)$.

\subsection{Flows with multiplicative noise}
\label{sec-mult_flows}

Next, we consider the case of multiplicative noise when \mbox{$d=1$}.
Consider the fast-slow system
\begin{align} \label{eq-h}
\nonumber \dotx & = \epsilon^{-1}h(\x)f_0(\y)+f(\x,\y), \quad \x(0)=\xi \\
\doty & = \epsilon^{-2}g(\y), \quad \y(0)=\eta,
\end{align}
with $\x\in\R$, $\y\in\R^\ell$.

\begin{thm}  \label{thm-mult}
(a) Assume that $g$, $f_0$, $f$, $F$, $W$ and $\Sigma$ are as in Theorem~\ref{thm-reg}
(but with the restriction that $d=1$ and $\Sigma$ is denoted by $\sigma^2>0$).
Suppose that $h:\R\to\R$ is $C^1$ and nonvanishing.

Let $\xi\in\R$ and consider the Stratonovich SDE
\begin{align} \label{eq-SDE_mult}
dX=\sigma h(X)\circ dW+F(X)\,dt, \quad X(0)=\xi.
\end{align}

Let $\x,\y$ denote the solutions to the fast-slow system~\eqref{eq-h}.
Assume that solutions $X$ to this SDE exist on $[0,\infty)$
with probability one.
Then $\x\to_w X$ in $C([0,\infty),\R)$ as $\epsilon\to0$.\\[-1.75ex]

\noindent(b)
More generally, assume the above set up but without the assumption that $h$
is nonvanishing.  Suppose that we can write $h=1/r'$ on an interval containing $\xi$.   Write $\tilde F=(F/h)\circ r^{-1}$ where defined.
Suppose that solutions $Z$ to the SDE~\eqref{eq-SDE_mult2_map} exist on $[0,\infty)$
with probability one.

Then solutions $X$ to the SDE~\eqref{eq-SDE_mult} exist on $[0,\infty)$ with probability one, and $\x\to_w X$ in $C([0,\infty),\R)$ as $\epsilon\to0$.
\end{thm}

Theorem~\ref{thm-mult} is proved by reducing it to Theorem~\ref{thm-reg}.
We require a preliminary elementary result.

\begin{prop} \label{prop-Ito}
Let $r:\R\to\R$ be a $C^2$ diffeomorphism.   Suppose that $W$ is
a one-dimensional unit Brownian motion and $\sigma^2>0$.
Consider the Stratonovich SDE 
\[
dX=\sigma(r'(X))^{-1}\circ dW+F(X)\,dt, \quad X(0)=\xi.
\]
  Then $X$ is a solution to this SDE if and only if
$Z=r(X)$ satisfies the SDE
\[
dZ=\sigma\,dW+\tilde F(Z)\,dt, \quad Z(0)=r(\xi),
\]
where $\tilde F=(r'F)\circ r^{-1}$.
\end{prop}

\begin{proof}  
Suppose that $X$ satisfies the first SDE.
Since the Stratonovich integral satisfies the usual chain rule,
$Z=r(X)$ satisfies
\[
dZ=r'(X)\circ dX= \sigma\,dW+r'(X)F(X)\,dt=
 \sigma\,dW+\tilde F(Z)\,dt.
\]
The converse direction is identical.
\end{proof}

\begin{pfof}{Theorem~\ref{thm-mult}}
Write $h=1/r'$ and let $Z=r(X)$ where $X$ satisfies the 
SDE~\eqref{eq-SDE_mult}.    
The assumptions on $h$ guarantee that $r$ is a $C^2$ diffeomorphism.
By Proposition~\ref{prop-Ito}, $Z$ satisfies the SDE~\eqref{eq-SDE_mult2_map}.

Next, let
$\z=r\circ \x$.  The 
$\dotx$ equation in~\eqref{eq-h} becomes
\begin{align} \label{eq-ODE_z}
\dotz(t)  =
\epsilon^{-1}f_0(y^{(\epsilon)}(t))+\tilde f(\z(t),\y(t)),
\end{align}
where 
\begin{align*}
\tilde f(z,y)=
f(r^{-1}z,y)/h(r^{-1}z).
\end{align*}

Since $\tilde F(z)=\int_\Lambda \tilde f(z,y)\,d\mu(y)$ and
$\tilde f$ is locally Lipschitz, we are now in the situation of
Theorem~\ref{thm-reg},
and it follows that solutions of~\eqref{eq-ODE_z}
converge weakly to solutions of~\eqref{eq-SDE_mult2_map}.
That is, $r\circ\x\to_wr\circ X$.
Applying the $C^1$ map
$r^{-1}$, it follows from the continuous mapping theorem that
$\x\to_w X$ as required.
\end{pfof}

\begin{rmk}\label{rmk-higher}  As mentioned in the introduction, this result has a restricted extension to higher dimensions.
Consider the fast-slow equations~\eqref{eq-h} with $\x\in\R^d$
and suppose that $h:\R^d\to L(\R^d,\R^d)$ can be written as 
$h=(Dr)^{-1}$ for some $C^2$ diffeomorphism $r:\R^d\to\R^d$.
Then the substitution $\z=r(\x)$ yields the equation~\eqref{eq-ODE_z}
with $\tilde f(z,y)=h(r^{-1}z)^{-1}f(r^{-1}z,y)$.
Again it follows from Theorem~\ref{thm-reg} that $\z\to_w Z$
where $dZ=\sqrt\Sigma\,dW + \tilde F(Z)\,dt$, 
$\tilde F(z)=h(r^{-1}z)^{-1}F(r^{-1}z)$.
The change of variables formulas for Stratonovich SDEs shows that
$X=r^{-1}(Z)$ satisfies 
$dX=\sqrt\Sigma h(X)\circ dW+F(X)\,dt$.
Again $\x\to_w X$ by the continuous mapping theorem.
\end{rmk}

\section{Proofs of the results for maps}
\label{sec-maps}

In this section we prove the discrete time homogenization results
stated in the introduction.     By the argument in the proof of Theorem~\ref{thm-reg}, we may suppose from the outset that all relevant
Lipschitz constants are uniform on the whole of $\R^d$ and that $\lim_{\epsilon\to0} f(x,y,\epsilon)=f(x,y,0)$ uniformly on the whole of $\R^d\times\Lambda$.

Throughout, it is more convenient to work with the piecewise constant function
$\tilde x^{(\epsilon)}(t)=\x([t\epsilon^{-2}])$ rather than with the linearly interpolated function $\hat x^{(\epsilon)}(t)$.
(Here $[t\epsilon^{-2}]$ denotes the integer part of $t\epsilon^{-2}$.)
Since the process $\tilde x^{(\epsilon)}(t)=\x([t\epsilon^{-2}])$ is not continuous, we can no longer work within $C([0,\infty),\R^d)$.  Instead we 
  prove weak convergence 
in the space $D([0,\infty),\R^d)$ of c\`adl\`ag functions (right-continuous functions with left-hand limits, see for example~\cite[Chapter~3]{Billingsley1999})
with the supremum norm.

It is clear that
$\sup_{t\in[0,T]}|\hat x^{(\epsilon)}(t)-\tilde x^{(\epsilon)}(t)|\to0$ as $\epsilon\to0$.
Hence weak convergence of $\tilde x^{(\epsilon)}$ in
$D([0,\infty),\R^d)$ is equivalent
to weak convergence of $\hat x^{(\epsilon)}$ 
in $D([0,\infty),\R^d)$.
This in turn is equivalent to 
weak convergence of $\hat x^{(\epsilon)}$ 
in $C([0,\infty),\R^d)$ (see for example the last line of p.~124 of~\cite{Billingsley1999}).

\subsection{Proof of Theorem~\ref{thm-add_map}}
\label{sec-proof}

Write $d(\epsilon)=\sup_{x\in\R^d,y\in\Lambda}|f(x,y,\epsilon)-f(x,y,0)|$, so
$\lim_{\epsilon\to0}d(\epsilon)=0$.
First note that 
\[
\x(n)=\xi+\epsilon\sum_{j=0}^{n-1}f_0(y(j))+
\epsilon^2\sum_{j=0}^{n-1}f(\x(j),y(j),\epsilon).
\]
Hence 
\begin{align*}
\tilde x^{(\epsilon)}(t)& =\xi+W^{(\epsilon)}(t)+ 
\epsilon^2\sum_{j=0}^{[t\epsilon^{-2}]-1}f(x^{(\epsilon)}(j),y(j),0)
+K_1^{(\epsilon)}(t) \\ & =
\xi+W^{(\epsilon)}(t)+ K_1^{(\epsilon)}(t)+
K_2^{(\epsilon)}(t)+
\epsilon^2\sum_{j=0}^{[t\epsilon^{-2}]-1}F(\tilde x^{(\epsilon)}(\epsilon^2j)).
\end{align*}
where 
$W^{(\epsilon)}(t)=\epsilon\sum_{j=0}^{[t\epsilon^{-2}]-1}f_0(y(j))$,
$|K_1^{(\epsilon)}(t)|\le Td(\epsilon)$,
 and
\[
K_2^{(\epsilon)}(t)=
\epsilon^2\sum_{j=0}^{[t\epsilon^{-2}]-1}
\bigl(f(\x(j),y(j),0)-F(\x(j))\bigr).
\]
For $t$ an integer multiple of $\epsilon^2$,
the term $\epsilon^2\sum_{j=0}^{[t\epsilon^{-2}]-1}F(\tilde x^{(\epsilon)}(\epsilon^2j))$ is the Riemann sum of a piecewise constant function and is precisely
$\int_0^t F(\tilde x^{(\epsilon)}(s))\,ds$.
For general $t$, 
\[
\epsilon^2\sum_{j=0}^{[t\epsilon^{-2}]-1}F(\tilde x^{(\epsilon)}(\epsilon^2j))=
\int_0^t F(\tilde x^{(\epsilon)}(s))\,ds+K_3^{(\epsilon)}(t),
\]
where $K_3^{(\epsilon)}(t)\le \epsilon^2|f|_\infty$.
Altogether,
\[
\tilde x^{(\epsilon)}(t)=\xi+W^{(\epsilon)}(t)+K^{(\epsilon)}(t)+ \int_0^t F(\tilde x^{(\epsilon)}(s))\,ds,
\]
where $K^{(\epsilon)}=K_1^{(\epsilon)}+ K_2^{(\epsilon)}+ K_3^{(\epsilon)}$.

By exactly the same argument as in~\cite{MS11}, 
$K_2^{(\epsilon)}\to0$ in probability in $D([0,T],\R^d)$.
(For convenience the proof is reproduced in Appendix~\ref{sec-avg}.)
It follows that $W^{(\epsilon)}+
K^{(\epsilon)}\to \sqrt\Sigma W$ in $D([0,T],\R^d)$.
Now consider the continuous map $\mathcal{G}:D([0,T],\R^d)\to D([0,T],\R^d)$ 
given by $\mathcal{G}(u)=v$ where $v(t)=\xi+u(t)+ \int_0^t F(v(s))\,ds$.
Then $\tilde x^{(\epsilon)}=\mathcal{G}(W^{(\epsilon)}+K^{(\epsilon)})$,
so $\tilde x^{(\epsilon)}\to_w \mathcal{G}(\sqrt\Sigma W)=X$.
This completes the proof of Theorem~\ref{thm-add_map}.

\subsection{Proof of Theorem~\ref{thm-mult_map} and Proposition~\ref{prop-mult_map}}

\label{sec-mult_map}

Again, we reduce to the case where $\tilde F$ is globally Lipschitz and uniformly
continuous at $\epsilon=0$.   Also, we may suppose that $r''$ is uniformly continuous.

Write $Z=r(X)$ where $X$ satisfies the SDE~\eqref{eq-SDE_mult_map}.    
By Proposition~\ref{prop-Ito}, $Z$ satisfies 
\begin{align} \label{eq-SDE_Z_map}
dZ=\sigma dW+\tilde F(Z)\,dt,
\end{align}
where
\begin{align*}
 \tilde F(z)  
& =r'(r^{-1}z) F(r^{-1}z)
-\frac12 h'(r^{-1}z)\int_\Lambda f_0^2\,d\mu
\\ & =r'(r^{-1}z) F(r^{-1}z)
+\frac12r''(r^{-1}z) [r'(r^{-1}z)]^{-2}\int_\Lambda f_0^2\,d\mu.
\end{align*}

Define $\z(n)= r(\x(n))$.
Using Taylor's theorem to expand the $C^2$ map $r$, we obtain
\begin{align} \label{eq-r}
\nonumber & \z(n+1)-\z(n)  =  r(x^{(\epsilon)}(n+1))-r(x^{(\epsilon)}(n)) \\[.75ex]
 & \qquad\qquad  = r'(\x(n))(\x(n+1)-\x(n)) 
\\[.75ex]\nonumber  & 
 \qquad\qquad\qquad
\quad + \textstyle{\frac12} r''(\x(n))(\x(n+1)-\x(n))^2 
+o((\x(n+1)-\x(n))^2),
\end{align} 
where the last term is uniformly $o(\epsilon^2)$ since
$r''$ is uniformly continuous.

Substituting for $\x(n+1)-\x(n)$ using equation~\eqref{eq-map_mult},
equation~\eqref{eq-r} becomes
\begin{align*}
 z^{(\epsilon)}(n+1)-
z^{(\epsilon)}(n)  & = \epsilon f_0(y(n))+\epsilon^2 
\Bigl\{r'(\x(n))f(\x(n),y(n),0) \\ & \qquad\qquad+\frac12 r''(\x(n))[r'(\x(n))]^{-2}f_0(y(n))^2+o(1)\Bigl\}
\\ & = \epsilon f_0(y(n))+\epsilon^2 \tilde f(z^{(\epsilon)}(n),y(n),\epsilon),
\end{align*}
where
\[
\tilde f(z,y,\epsilon)=r'(r^{-1}z)f(r^{-1}z,y,0)+\frac12
r''(r^{-1}z)[r'(r^{-1}z)]^{-2}f_0(y)^2+o(1),
\]
uniformly in $z,y$ as $\epsilon\to0$.
Since $\tilde F(z)=\int_\Lambda\tilde f(z,y,0)\,d\mu(y)$, it
follows from Theorem~\ref{thm-add_map} that $\tilde z^{(\epsilon)}(t)=\z([t\epsilon^{-2}])$ 
converges weakly to solutions $Z$ of the SDE~\eqref{eq-SDE_Z_map}.
Applying the $C^1$ map
$r^{-1}$, it follows from the continuous mapping theorem that
$\tilde x^{(\epsilon)}\to_w X$ 
(and hence $\tilde x^{(\epsilon)}\to_w X$) as required.

\section{SDEs driven by stable L\'evy processes}
\label{sec-Levy}

In this section, we consider the situation where the fast dynamics is not sufficiently chaotic to support the WIP.   With reference to Remark~\ref{rmk-Young}, this occurs when the Young tower~\cite{Young99} modelling the map (or the Poincar\'e map in the case of flows) has nonsummable decay of correlations.  In this case, 
weak convergence to Brownian motion fails.
However there are instances where instead there is convergence to a stable
L\'evy process.

The prototypical examples are provided by Pomeau-Manneville intermittency 
maps~\cite{PomeauManneville80}.  For definiteness, consider the 
maps $g:[0,1]\to[0,1]$ studied by~\cite{LiveraniSaussolVaienti99}:
\begin{align}
g(y)=\begin{cases} y(1+2^\gamma y^\gamma), & y\in[0,\frac12)
\\ 2y-1, & y\in[\frac12,1] \end{cases}. \label{eq-PM}
\end{align}
For $\gamma\in(0,1)$, there exists a unique absolutely continuous invariant ergodic probability $\mu$, and correlations decay at the rate $n^{-(\gamma^{-1}-1)}$.
The attractor is $\Lambda=[0,1]$.

In particular, correlations are summable if and only if $\gamma<\frac12$,
and in this situation all of the results in the previous sections apply.
From now on, we suppose that $\gamma\in(\frac12,1)$.   
Suppose that $f_0:\Lambda\to\R^d$ is Lipschitz with
$\int_\Lambda f_0\,d\mu=0$, and assume further that $f_0(0)\neq0$.
Then Gou\"ezel~\cite{Gouezel04} proved that the central limit theorem fails and instead that
$n^{-\gamma}\sum_{j=0}^{n-1}f_0(y(j))$ converges in distribution to a 
stable law $Y$ of exponent $1/\gamma$.   
(More precisely, it follows from~\cite{Gouezel04} that if $c\in\R^d$ and 
$c\cdot f_0(0)\neq0$, then $n^{-\gamma}\sum_{j=0}^{n-1}c\cdot f_0(y(j))$ converges in distribution to a $1$-dimensional stable law of exponent $1/\gamma$.
Hence there is convergence in distribution to a $d$-dimensional random variable $Y$ and $c\cdot Y$ is stable of exponent $1/\gamma$ for all $c$.   By~\cite[Theorem 2.1.5(a) or (c)]{SamorodnitskyTaqqu94}, $Y$ is a $d$-dimensional stable distribution, and its exponent is $1/\gamma$ by~\cite[Theorem 2.1.2]{SamorodnitskyTaqqu94}.)

Let $G=G_{1/\gamma}$ denote the 
corresponding stable L\'evy process (independent and stationary increments with
$G(t)=_dt^\gamma Y$ for each $t$ and sample paths lying in $D([0,\infty),\R^d)$).   Then it follows from~\cite{MZsub} that
$W_n(t)=n^{-\gamma}\sum_{j=0}^{[nt]-1}f_0(y(j))$ converges weakly to $G$ in $D([0,\infty),\R^d)$ with the Skorokhod $\mathcal{M}_1$ topology.

We proceed to consider fast-slow systems where the fast dynamics satisfies weak convergence to a stable L\'evy process of exponent $1/\gamma$.
First, consider the fast-slow ODE
\begin{align*}
\dotx & =\epsilon^{\gamma-1}h(\x)f_0(\y)+f(\x,\y), \quad \x(0)=\xi \\
\doty & =\epsilon^{-1}g(\y), \quad \y(0)=\eta.
\end{align*}

If $h(x)\equiv1$, then exactly the same arguments as before yield that
$\x\to_w X$ where $X$ satisfies the SDE
\[
dX=dG+F(X)\,dt, \quad X(0)=\xi.
\]

If $h=1/r'$ is nontrivial and $d=1$, then the same argument as before
shows that $\x\to_w X$ where $X=r(Z)$ and $Z$ is the solution
of the SDE
\[
dZ=dG+\tilde F(Z)\,dt, \quad Z(0)=\xi,
\]
where $\tilde F=(F/h)\circ r^{-1}$.
Transforming back, it is immediate that $X$ satisfies the SDE
\begin{align} \label{eq-Marcus}
dX=h(X)\diamond dG+F(X)dt, \quad X(0)=\xi,
\end{align}
{\em provided} that the stochastic integral $h(X)\diamond dG$ satisfies the
usual chain rule.  For L\'evy processes, it turns out that neither the 
It\^o nor Stratonovich interpretation is suitable, and the correct integral
is due to Marcus~\cite{Marcus81} (see also~\cite{KurtzPardouxProtter95}).
See~\cite[p.~272]{Applebaum} for a discussion of the Marcus stochastic integral
and in particular for the chain rule~\cite[Theorem 4.4.28]{Applebaum}.

For maps, we consider
\[
\x(n+1)=\x(n)+\epsilon^\gamma h(\x(n))f_0(y(n))+\epsilon f(\x(n),y(n)),
\]
Set $\hat x^{(\epsilon)}(t)=\x([t\epsilon^{-1}])$.    If $h\equiv1$, we obtain again that
$\hat x^{(\epsilon)}$ converges weakly to solutions $X$ of the SDE
\[
dX=dG+F(X)\,dt, \quad X(0)=\xi.
\]
If $h$ is nontrivial and $d=1$, then we again obtain the Marcus SDE~\eqref{eq-Marcus}.
(Note that the second order expansion of $r$ yields terms of order $2\gamma$ which are negligible since $\gamma>\frac12$.  Hence the additional correction terms that arose in the discrete case for Brownian noise are absent for L\'evy noise.)

\section{Numerical validation}
\label{sec-numerics}

In this section, we illustrate Theorem~\ref{thm-mult_map} with a numerical simulation of a suitable fast-slow map.
For the fast dynamics we could consider a Pomeau-Manneville intermittent map
as in~\eqref{eq-PM}, with $\gamma\in[0,\frac12)$ so that the WIP is satisfied.  In order to satisfy the centering condition
$\int_\Lambda f_0\,d\mu=0$, it is more convenient to work
with the following modified version of the map in~\eqref{eq-PM}:
\begin{align}
g(y)=\begin{cases} y(1+2^\gamma y^\gamma), & y\in[0,\frac12)
\\ 1-2y, & y\in[\frac12,1] \\
-g(-y), & y\in[-1,0) \end{cases}. \label{eq-modPM}
\end{align}
Again, for $\gamma\in[0,1)$ there exists a unique absolutely continuous invariant ergodic probability measure $\mu$, Moreover, the WIP again holds for $\gamma\in[0,\frac12)$; we choose $\gamma=0.1$.  The attractor is $\Lambda=[-1,1]$.

Since the map $g:\Lambda\to\Lambda$ is odd, the probability measure $\mu$ is symmetric around the origin. Hence the condition $\int_\Lambda f_0\,d\mu=0$ is automatically satisfied provided $f_0:\Lambda\to\R$ is odd.
We choose 
\[
f_0(y)=y, \quad h(x)=x^\frac12, \quad
f(x,y,\epsilon)={\SMALL\frac12}({\SMALL\frac34}-x)y^2,
\]
so the slow dynamics is given by
\begin{align}
\x(n+1) &= \x(n) +\epsilon \x(n)^\frac12 y(n) +\epsilon^2{\SMALL\frac12}({\SMALL\frac34}-\x(n))y(n)^2.
\label{eq-slow}
\end{align}

According to Theorem~\ref{thm-mult_map}, rescaled solutions $\hat x^{(\epsilon)}(t)=\x(t\epsilon^{-2})$ of \eqref{eq-slow} converge weakly 
to solutions of the SDE
\begin{align} \nonumber
dX &= \sigma X^\frac12 \circ dW+\frac12\Bigl(\frac34-X\Bigr)\int_\Lambda y^2\,d\mu\, dt - \frac14\int_\Lambda y^2\,d\mu\,  dt 
 \\ \nonumber & = \sigma X^\frac12 \, dW+\frac12\Bigl(\frac34-X\Bigr)\int_\Lambda y^2\,d\mu\, dt + \frac14\Bigl(\sigma^2-\int_\Lambda y^2\,d\mu \Bigr) \, dt ,
 \\ & = \sigma X^\frac12 \, dW+\alpha(\beta-X)\, dt,
\label{ex_limit}
\end{align}
where $W$ is unit $1$-dimensional Brownian motion and 
\begin{align*}
& \SMALL \sigma^2=\int_\Lambda y(0)^2\,d\mu+2\sum_{n=1}^\infty \int_\Lambda y(n)y(0)\,d\mu=\lim_{n\to\infty}n^{-1}\int_\Lambda (\sum_{j=0}^{n-1}y(j))^2\,d\mu, \\
& \SMALL \alpha=\frac12\int_\Lambda y^2\,d\mu, \quad \beta=\frac14(1+\sigma^2/\alpha).
\end{align*}
This is the Cox-Ingersoll-Ross (CIR) model~\cite{CIR85,CIR85b} which has the closed form solution 
\begin{align} \label{eq-CIR}
X(t)=c(t)H(t),\quad c(t)=\frac{\sigma^2}{4\alpha}(1-e^{-\alpha t}),
\end{align}
where $H(t)$ is a noncentral $\chi$-squared distribution with $4\alpha\beta/\sigma^2$
degrees of freedom and noncentrality parameter 
$c(t)^{-1}e^{-\alpha t}\xi$.

A long time iteration of the map~\eqref{eq-modPM}, taking an ensemble average, yields the approximate values $\sigma^2=0.085$ and $\int_\Lambda y^2\,d\mu=0.319$,  and hence $\alpha=0.160$, $\beta=0.383$.

A consequence of weak convergence is convergence in distribution for each fixed $t>0$, namely that for any $a\in\R$, 
\[
\lim_{\epsilon\to0}\mu(\hat x^{(\epsilon)}(t)<a)= P(X(t)<a).
\]
We proceed to verify this result numerically with $t=10$
and the initial condition $\xi=\x(0)=X(0)=1$.

Computing the probability density function for $\hat x^{(\epsilon)}(10)$ from the numerical simulation of the full fast-slow system, and the limiting probability density function for $X(10)$ using the closed-form solution~\eqref{eq-CIR},
we obtain the results shown
in Figure~\ref{fig-1}.   We used ensembles consisting of $5,000,000$ realizations (though for the fast-slow system we found that $100,000$ realizations were ample).

In particular, Figure~\ref{fig-1} confirms our prediction regarding the drift term 
$\frac12(\frac34-X)\int_\Lambda y^2\,d\mu+ \frac14(\sigma^2-\int_\Lambda y^2\,d\mu)$ in the limiting SDE~\eqref{ex_limit}.   
Figure~\ref{fig-2} shows a comparison of this probability density function with those that would result from having the incorrect drift terms 
$\frac12(\frac34-X)\int_\Lambda y^2\,d\mu$ or
$\frac12(\frac34-X)\int_\Lambda y^2\,d\mu+ \frac14\sigma^2$
that arise when the limiting SDE is interpreted as being It\^o (as in the iid case) or Stratonovich (as in the continuous time case).

\begin{figure}
\centering
\includegraphics[width = 0.49\columnwidth, height = 5cm]{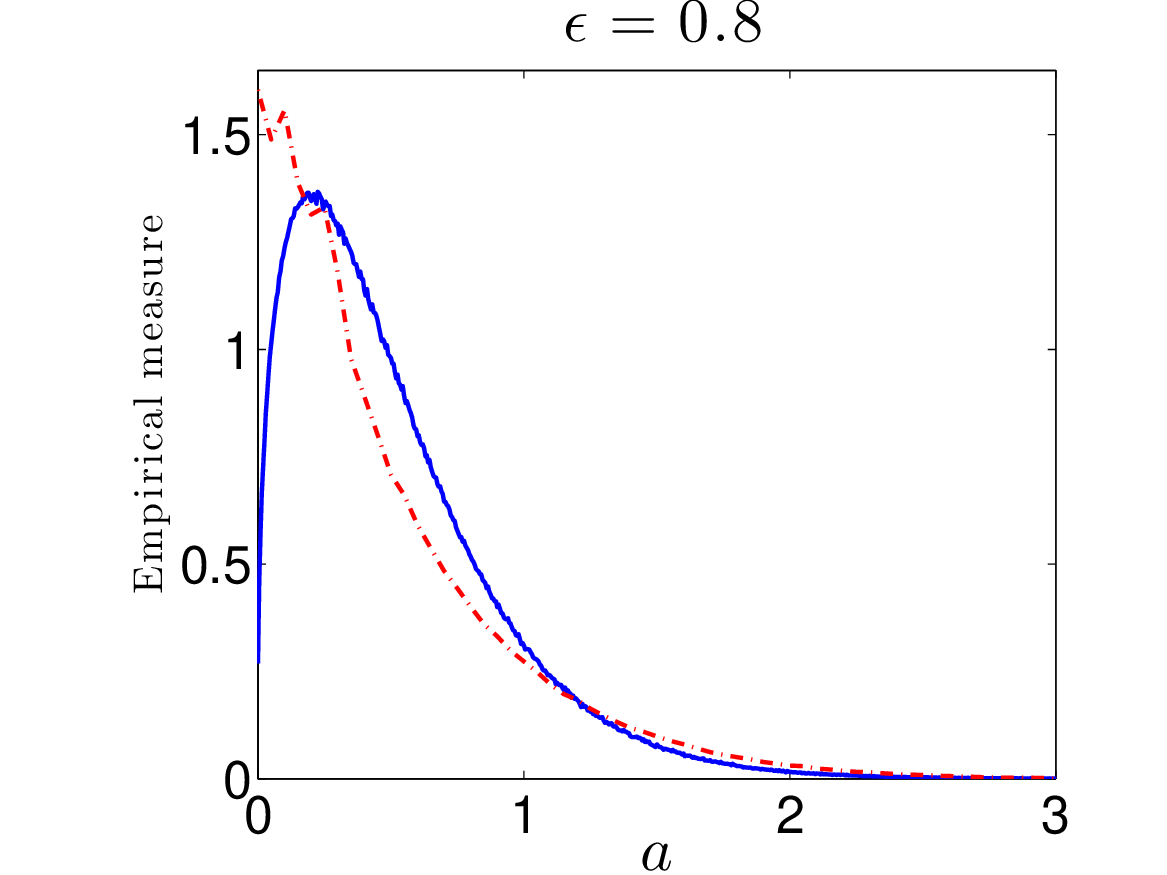}
\includegraphics[width = 0.49\columnwidth, height = 5cm]{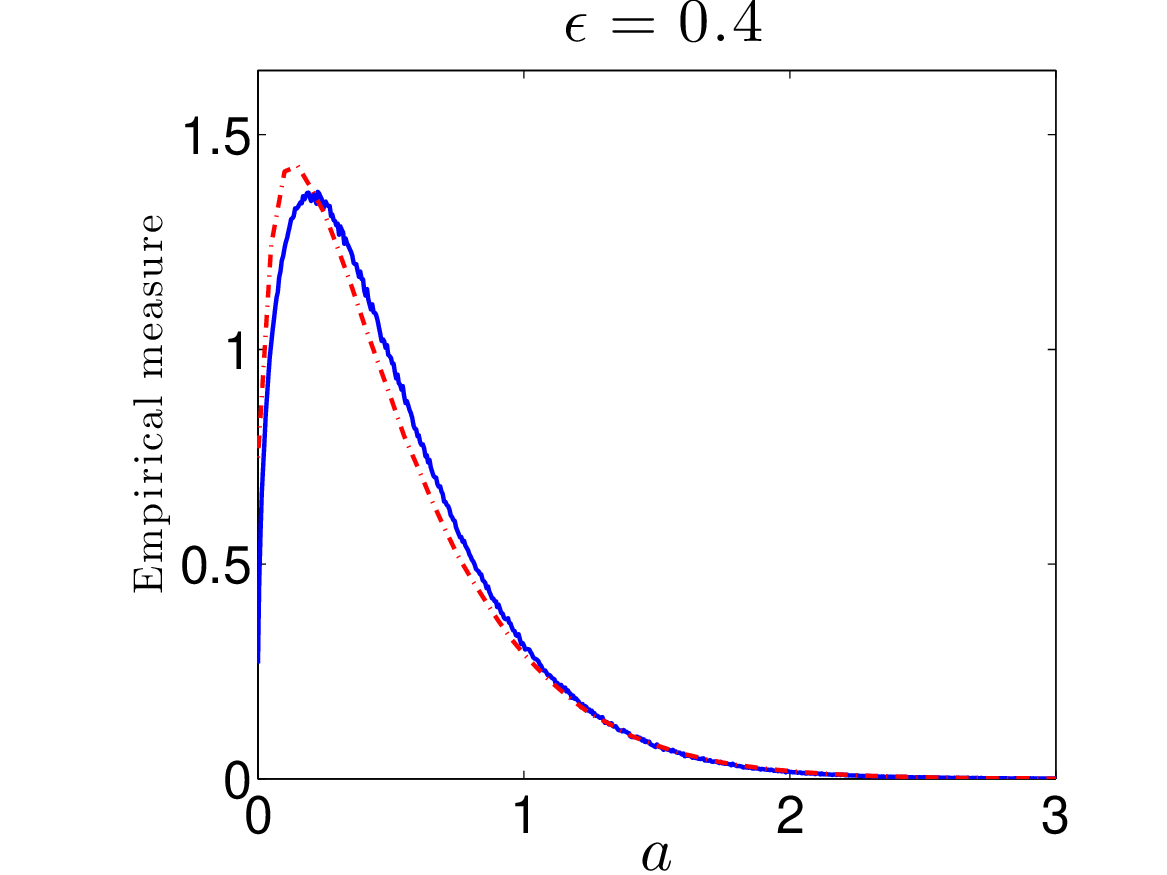}
\includegraphics[width = 0.49\columnwidth, height = 5cm]{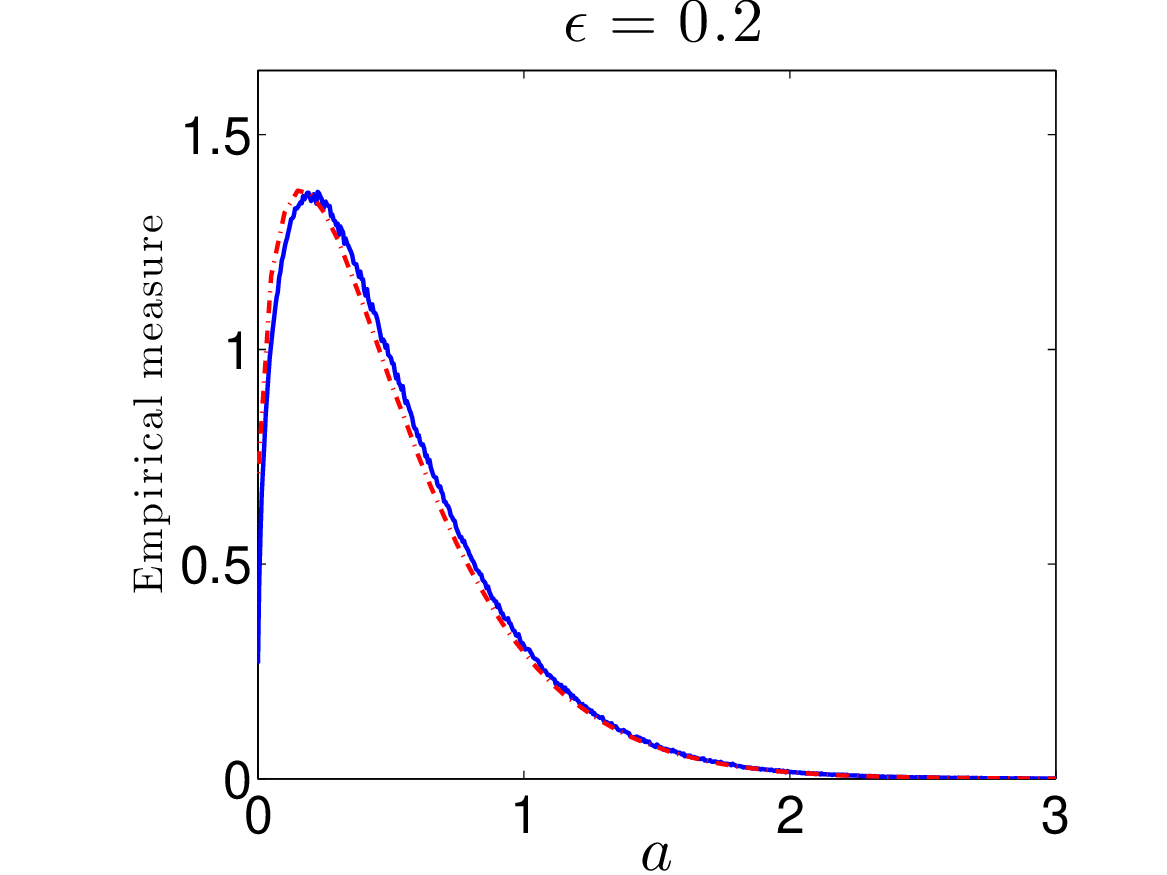}
\caption{Probability density function (empirical measure) at $t=10$ for the fast-slow map~\eqref{eq-modPM},~\eqref{eq-slow} with $\epsilon=0.8,0.4,0.2$ (dashed, red) and for the SDE limit~\eqref{ex_limit} (solid, blue).  We used ensembles consisting of $5,000,000$ realizations.}
\label{fig-1}
\end{figure}

\begin{figure}
\centering
\includegraphics[width = 0.98\columnwidth, height = 7cm]{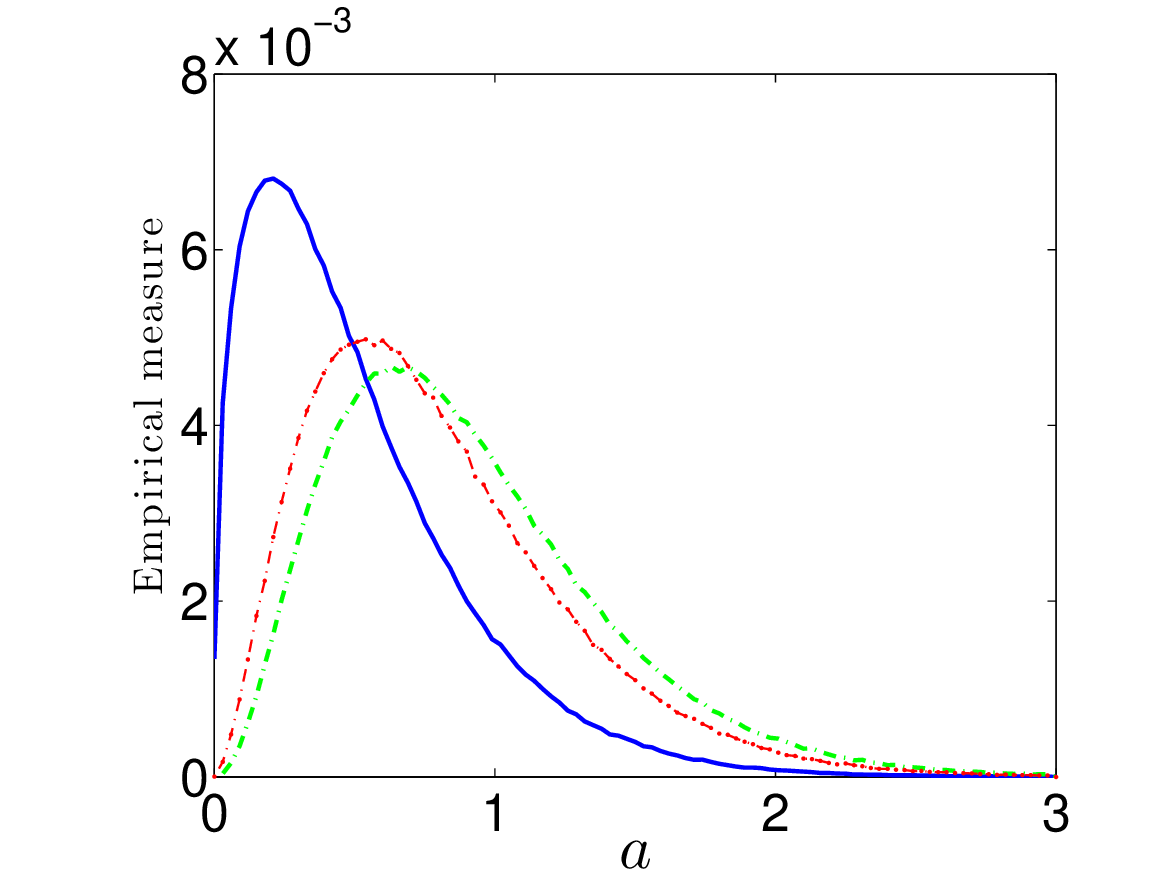}
\caption{Probability density function (empirical measure) at $t=10$ for the
limiting SDE with different drift terms corresponding to three different interpretations of the stochastic integral: the theoretically predicted drift term
in~\eqref{ex_limit} (left, blue) and those for the It\^o interpretation
(middle, red)
and the Stratonovich interpretation (right, green).
We used ensembles consisting of $5,000,000$ realizations.}
\label{fig-2}
\end{figure}

We also present numerical evidence for convergence of first moments
$\E(|\hat x^{(\epsilon)}(t)|)$.   This is not a direct consequence of Theorem~\ref{thm-mult_map}.   However, for each $t>0$,  Theorem~\ref{thm-mult_map} together with boundedness (as $\epsilon$ varies) of a higher moment implies~(eg.~\cite[Exercise~2.5, p.~86]{Durrett} that 
$\E(|\hat x^{(\epsilon)}(t)|)$ converges, as $\epsilon\to0$, to
\begin{align} \label{eq-mean_CIR}
\E(X(t))=\xi e^{-\alpha t}+\beta(1-e^{-\alpha t}).  
\end{align}
We verified numerically that $\E((\hat x^{(\epsilon)}(t))^2)$ is convergent and hence bounded, implying convergence of 
$\E(|\hat x^{(\epsilon)}(t)|)$ as demonstrated in Figure~\ref{fig-3}.

\begin{figure}
\centering
\includegraphics[width = 0.49\columnwidth, height = 7cm]{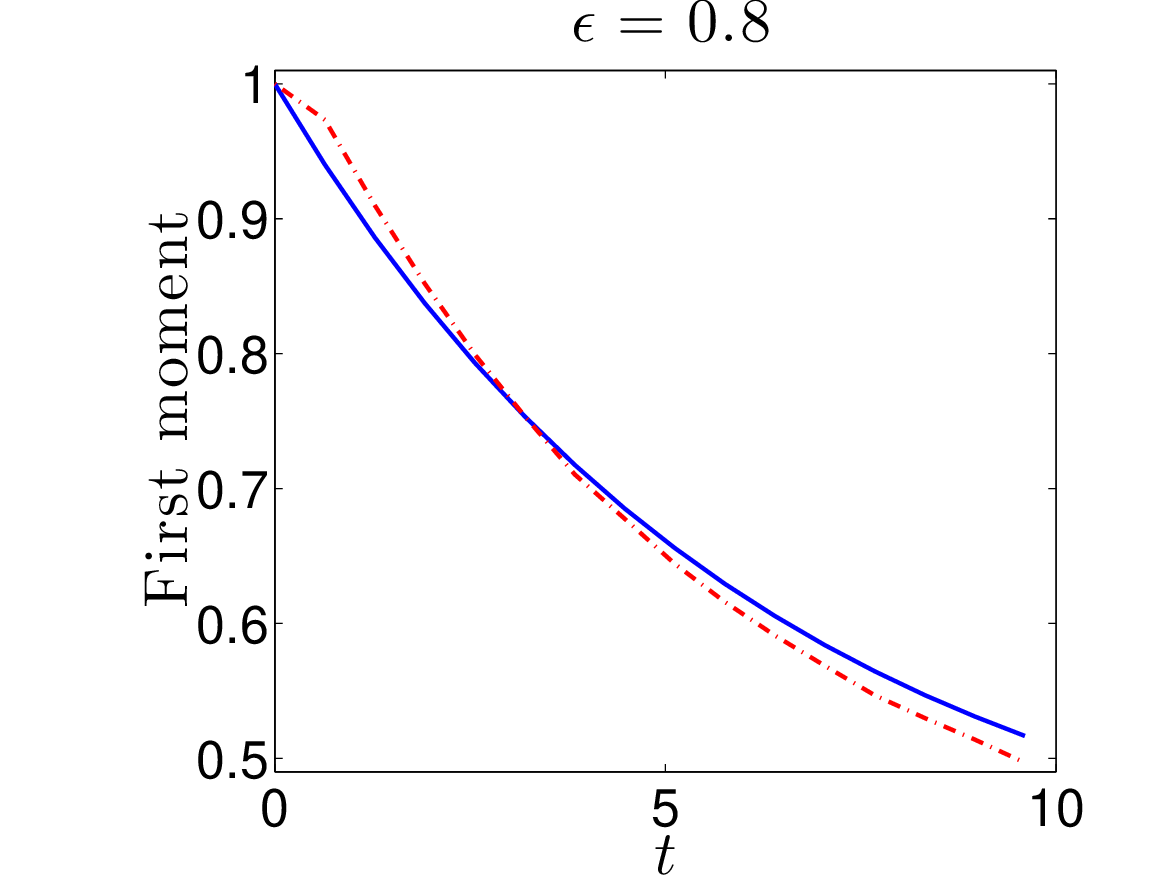}
\includegraphics[width = 0.49\columnwidth, height = 7cm]{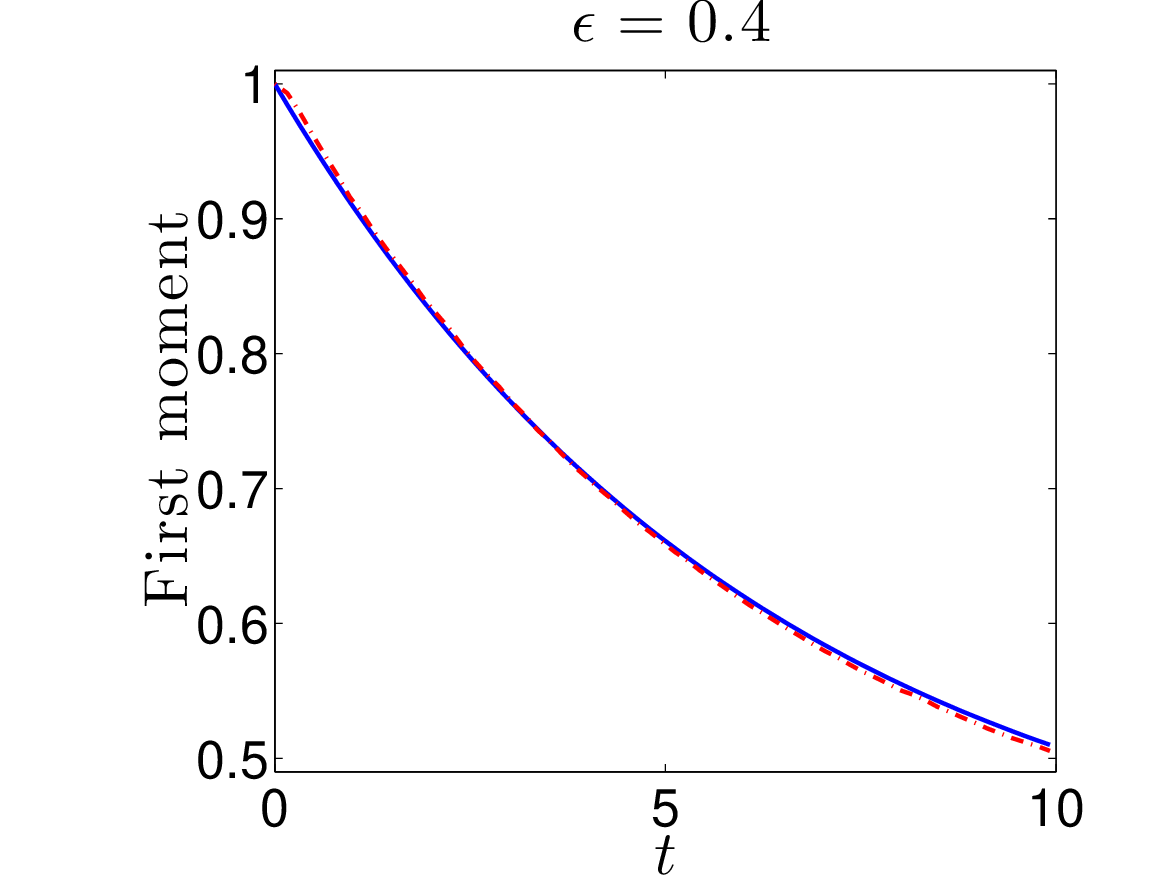}
\includegraphics[width = 0.49\columnwidth, height = 7cm]{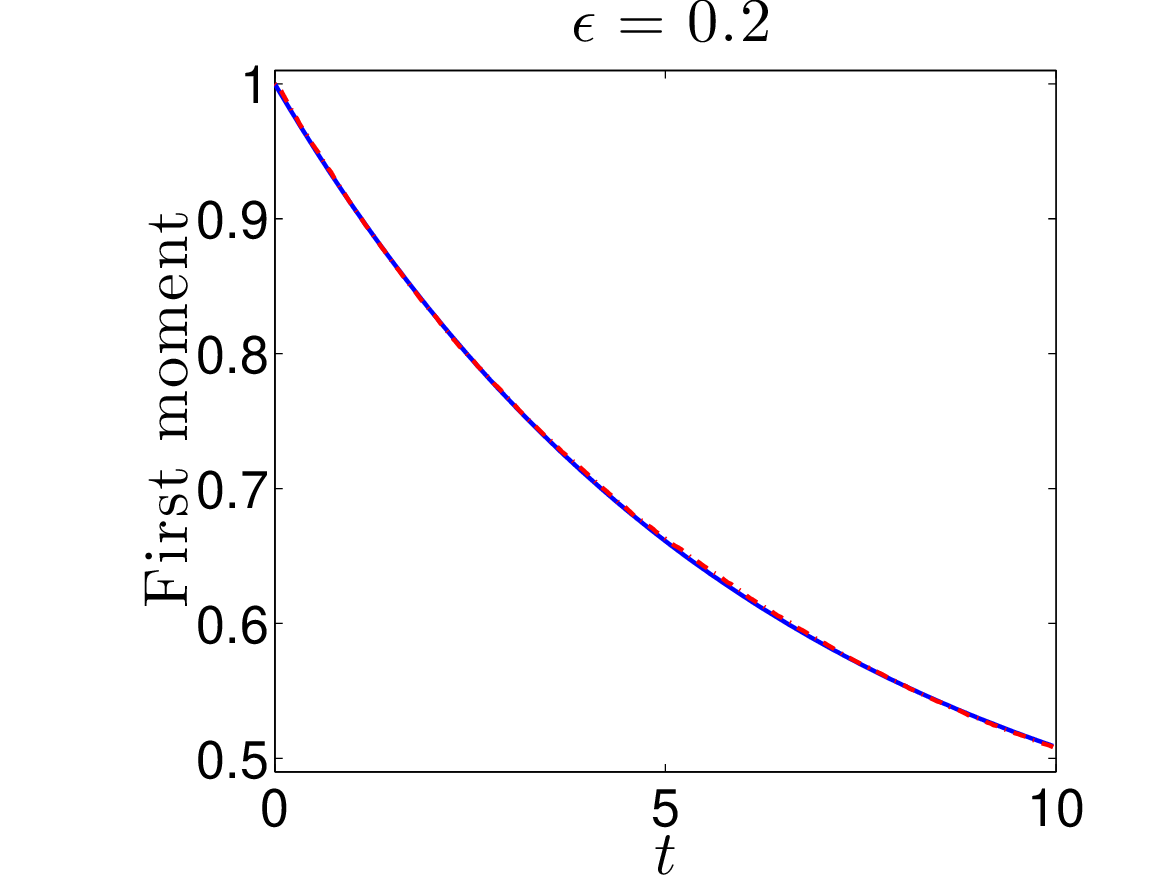}
\caption{First moment $0\le t\le 15$ for the fast-slow map~\eqref{eq-modPM},~\eqref{eq-slow} with $\epsilon=0.8,0.4,0.2$ (dashed, red) and for the SDE limit~\eqref{eq-mean_CIR} (solid, blue).  We used ensembles consisting of $100,000$ realizations for the fast-slow map.}
\label{fig-3}
\end{figure}

\section{Conclusions}
\label{sec-conc}

The paper~\cite{MS11} set out a programme for a rigorous investigation of homogenisation for fast-slow deterministic systems under very mild assumptions on the fast dynamics.    In this paper, we have extended these results in a number of ways:
\begin{itemize}
\item[1.]  Extension from continuous time to discrete time.
\item[2.]  Incorporation of one-dimensional multiplicative noise.
\item[3.]  Inclusion of situations where the SDE is driven by a stable
L\'evy process.
\item[4.]  Relaxation of regularity assumptions, in particular the requirement in~\cite{MS11} that certain vector fields are uniformly Lipschitz.
\end{itemize}

Items 2--4 require interpretation of certain stochastic integrals.   In the 
case of multiplicative noise in continuous time, the stochastic integrals are
of Stratonovich type as would be expected.    For discrete time, it was pointed out in~\cite{GivonKupferman04} that the integrals are neither It\^{o} nor Stratonovich.   We recover their result in a much more general context.    For multiplicative noise in the situation of item 3, where the SDE is L\'evy, we obtain Marcus stochastic integrals for both discrete and continuous time.

Important directions of future research include the analysis of higher-dimensional multiplicative noise and of fully coupled systems where the fast dynamics depends on the slow variables.

\appendix
\section{Averaging argument}
\label{sec-avg}

Let $K_2^{(\epsilon)}$ be the expression appearing in the proof of Theorem~\ref{thm-add_map} in Subsection~\ref{sec-proof}.
For completeness, we give the details of the argument that
$K_2^{(\epsilon)}\to0$ in probability in $D([0,T],\R^d)$.
The proof is identical to that in the continuous time context in~\cite{MS11}.
(Note that the published version of~\cite{MS11} contains an error that is corrected in the updated version on arXiv.  The published version of our paper replicated this error which is corrected in identical fashion below.)

\begin{lemma} \label{lem-K2}
$K_2^{(\epsilon)}\to0$ in probability in $D([0,T],\R^d)$.
\end{lemma}

\begin{proof}
Define $\tilde f(x,y)=f(x,y,0)-F(x)$ and note that $|\tilde f|_\infty\le 2|f|_\infty$
and $\Lip(\tilde f)\le 2\Lip(f)$.   Then
$K_2^{(\epsilon)}(t) =\epsilon^2\sum_{0\le j<[t\epsilon^{-2}]} \tilde f(\x(j),y(j))$.
Let $N=[t/\epsilon^{3/2}]$ and write $K_2^{(\epsilon)}(t)=K_2^{(\epsilon)}(N\epsilon^{3/2})+I_0$
where $I_0=\epsilon^2\sum_{N\epsilon^{-1/2}\le j<t\epsilon^{-2}} \tilde f(\x(j),y(j))$.
We have
\begin{align} \label{eq-maps_I0}
|I_0|\le (t-N\epsilon^{3/2})|\tilde f|_\infty \le 2|f|_\infty \epsilon^{3/2}.
\end{align}
We now estimate $K_2^{(\epsilon)}(N\epsilon^{3/2})$ as follows:
\begin{align*}
K_2^{(\epsilon)}(N\epsilon^{3/2})&= \epsilon^2\sum_{n=0}^{N-1} 
\sum_{n\epsilon^{-1/2}\le j<(n+1)\epsilon^{-1/2}} 
\tilde f(\x(j),y(j))\\
& =\epsilon^2\sum_{n=0}^{N-1}
\sum_{n\epsilon^{-1/2}\le j<(n+1)\epsilon^{-1/2}} 
\bigl(\tilde f(\x(j),y(j))-\tilde f(\x(n\epsilon^{-1/2}),y(j))\bigr) \\ &\qquad\qquad +\epsilon^2\sum_{n=0}^{N-1}
\sum_{n\epsilon^{-1/2}\le j<(n+1)\epsilon^{-1/2}} 
\tilde f(\x(n\epsilon^{-1/2}),y(j)) \\
& = I_1+I_2.
\end{align*}

For $n\epsilon^{-1/2}\le j<(n+1)\epsilon^{-1/2}$, we have
$|\x(j)-\x(n\epsilon^{-1/2})| \le (\epsilon|f_0|_{\infty}+\epsilon^2|f|_{\infty})\epsilon^{-1/2}$.
Hence
\begin{align} \label{eq-maps_I1}
|I_1| & \le N\epsilon^{3/2} \Lip(\tilde f)(|f_0|_{\infty}+|f|_{\infty})\epsilon^{1/2}  \le 
2\Lip(f)(|f_0|_{\infty}+|f|_{\infty})T\epsilon^{1/2}.
\end{align}

Next,
\begin{align*}
I_2 & =\epsilon^2\sum_{n=0}^{N-1} 
\sum_{n\epsilon^{-1/2}\le j<(n+1)\epsilon^{-1/2}} 
\tilde f(\x(n\epsilon^{-1/2}),y(j))
=\epsilon^{3/2}\sum_{n=0}^{N-1}J_n,
\end{align*}
where
\begin{align*}
J_n  &= \epsilon^{1/2}
\sum_{n\epsilon^{-1/2}\le j<(n+1)\epsilon^{-1/2}} 
\tilde f(\x(n\epsilon^{-1/2}),y(j)).
\end{align*}
Hence
\begin{align} \label{eq-J}
|I_2|\le \epsilon^{3/2}\sum_{n=0}^{[T\epsilon^{-3/2}]-1}|J_n|.
\end{align}

For $u\in\R^d$ fixed, we define
\[
\tilde J_n(u)=\epsilon^{1/2}  \sum_{n\epsilon^{-1/2}\le j<(n+1)\epsilon^{-1/2}} 
\tilde f(u,y(j))
=\epsilon^{1/2}\sum_{n\epsilon^{-1/2}\le j<(n+1)\epsilon^{-1/2}} 
A_u\circ g^j
\]
where $A_u(y)=\tilde f(u,y)$.
Note that $\tilde J_n(u)=\tilde J_0(u)\circ g^{[n\epsilon^{-1/2}]}$,
and so
$\E|\tilde J_n(u)| =\E|\tilde J_0(u)|$.
By the ergodic theorem,
$\E|\tilde J_0(u)|\to0$ as $\epsilon\to0$ for each $u$.

Let $Q>0$ and write $I_2=M_{Q,1}+M_{Q,2}$ where
\begin{align*}
& M_{Q,1}=I_21_{B_\epsilon(Q)}, \quad
M_{Q,2}=I_21_{B_\epsilon(Q)^c}, \quad
 B_\epsilon(Q)=\bigl\{\max_{[0,T]}|\x|\le Q\bigr\}.
\end{align*}

For any $a>0$, there exists a finite subset $S\subset\R^d$ such that
$\operatorname{dist}(x,S)\le a/(2\Lip(f))$ for any $x$ with $|x|\le Q$.
Then for all $n\ge0$, $\epsilon>0$,
\[
1_{B_\epsilon(Q)}|J_n|\le \sum_{u\in S}|\tilde J_n(u)|+a.
\]
Hence by~\eqref{eq-J},
\begin{align*}
 \E\max_{[0,T]}|M_{Q,1}|
 & \le \epsilon^{3/2}\sum_{n=0}^{[T\epsilon^{-3/2}]-1}
\sum_{u\in S}\E|\tilde J_n(u)|+Ta
\\ & = \epsilon^{3/2}\sum_{n=0}^{[T\epsilon^{-3/2}]-1}
\sum_{u\in S}\E|\tilde J_0(u)|+Ta
  \\ & \le T \sum_{u\in S}\E|\tilde J_0(u)|+Ta.
\end{align*}
Since $a>0$ is arbitrary, we obtain for each fixed $Q$ that
$\max_{[0,T]}|M_{Q,1}|\to0$ in $L^1$, and hence in probability, as $\epsilon\to0$.

Next, since $\x-W^{(\epsilon)}$ is bounded on $[0,T]$, for $Q$ sufficiently large
\[
\mu\bigl\{\max_{[0,T]}|M_{Q,2}|>0\bigr\}\le 
\mu\bigl\{\max_{[0,T]}|\x|\ge Q\bigr\}\le 
\mu\bigl\{\max_{[0,T]}|W^{(\epsilon)}|\ge Q/2\bigr\}.
\]
Fix $c>0$.  Increasing $Q$ if necessary, we can arrange that
$\mu\{\max_{[0,T]}|\sqrt\Sigma W|\ge Q/2\}<c/4$.
By the continuous mapping theorem,
$\max_{[0,T]}|W^{(\epsilon)}|\to_d \max_{[0,T]}|\sqrt\Sigma W|$.
Hence there exists $\epsilon_0>0$ such that
$\mu\{\max_{[0,T]}|W^{(\epsilon)}|\ge Q/2\}<c/2$ for all $\epsilon\in(0,\epsilon_0)$.
For such $\epsilon$,
\[
\mu\bigl\{\max_{[0,T]}|M_{Q,2}|>0\bigr\}<c/2.
\]
Shrinking $\epsilon_0$ if necessary, we also have that
$\mu\{\max_{[0,T]}|M_{Q,1}|>c/2\}<c/2$.
Hence $\mu\{\max_{[0,T]}|I_2|>c\}<c$, and so
$\max_{[0,T]}|I_2|\to0$ in probability.
Combining this with
estimates~\eqref{eq-maps_I0} and~\eqref{eq-maps_I1},
we obtain that 
$\max_{[0,T]}|K_2^{(\epsilon)}|\to0$ in probability as required.
\end{proof}

\paragraph{Acknowledgements} 
The research of GAG was supported in part by ARC grant FT 0992214.
The research of IM was supported in part by EPSRC Grant EP/F031807/1 held at the University of Surrey.
IM is grateful to the hospitality at the University of Sydney where this research commenced in 2011.
We are grateful to Ben Goldys, David Kelly and especially Andrew Stuart for helpful discussions, and
also to the referees for several helpful comments and suggestions.

\end{document}